\documentclass[11pt,reqno]{amsart}

\usepackage{amsmath,amsthm,amsfonts,amssymb}
\numberwithin{equation}{section}

\usepackage{verbatim} 
\usepackage{xspace} 
\usepackage{amssymb}
\usepackage{amsmath}
\usepackage{amsthm} 
\usepackage[latin1]{inputenc}
\usepackage{graphicx}
\usepackage{amscd,amssymb,euscript,graphics}

\makeindex

\newcommand{\B}{\ensuremath{\mathbb{B}}}

\newcommand{\Z}{\ensuremath{\mathbb{Z}}}

\newcommand{\R}{\ensuremath{\mathbb{R}}}

\newcommand{\F}{\ensuremath{\mathbb{F}}}

\newcommand{\mh}{\ensuremath{\mathbb{H}}}

\newtheorem{teo}{Theorem}[section]

\newtheorem{df}[teo]{Definition}
\newtheorem{ej}[teo]{Example}
\newtheorem{obs}[teo]{Remark}

\newcount\theTime
\newcount\theHour
\newcount\theMinute
\newcount\theMinuteTens
\newcount\theScratch
\theTime=\number\time
\theHour=\theTime
\divide\theHour by 60
\theScratch=\theHour
\multiply\theScratch by 60
\theMinute=\theTime
\advance\theMinute by -\theScratch
\theMinuteTens=\theMinute
\divide\theMinuteTens by 10
\theScratch=\theMinuteTens
\multiply\theScratch by 10
\advance\theMinute by -\theScratch

\def\today{{\number\day\space
 \ifcase\month\or
  January\or February\or March\or April\or May\or June\or
  July\or August\or September\or October\or November\or December\fi
 \space\number\year}}

\begin{document}

\title[Asymptotic traffic flow in a Hyperbolic Network]{Asymptotic traffic flow in a Hyperbolic Network: Non-uniform Traffic}

\author[Yuliy Baryshnikov and Gabriel H. Tucci]
{Yuliy Baryshnikov and Gabriel H. Tucci}

\address{Bell Laboratories Alcatel--Lucent,
Murray Hill, NJ 07974, USA}
\email{ymb@alcatel-lucent.com}
\email{gabriel.tucci@alcatel-lucent.com}

\begin{abstract}
In this work we study the asymptotic traffic flow in Gromov's hyperbolic graphs when the traffic decays exponentially with the distance. We prove that under general conditions, there exists a phase transition between local and global traffic. More specifically, assume that the traffic rate between two nodes $u$ and $v$ is given by $R(u,v)=\beta^{-d(u,v)}$ where $d(u,v)$ is the distance between the nodes. Then there exists a constant $\beta_c$ that depends on the geometry of the network such that if $1<\beta<\beta_c$ the traffic is global and there is a small set of highly congested nodes called the core. However, if $\beta>\beta_c$ then the traffic is essentially local and the core is empty which implies very small congestion.
\end{abstract}

\maketitle

\section{Introduction}
\vspace{0.3cm}
\noindent The structure of networks has been mainly the domain of a branch of discrete mathematics known as graph theory. Some basic ideas, used later by physicists, were proposed in 1959 by the Hungarian mathematician Paul Erd\"os and his collaborator R\'enyi.  Graph theory has witnessed many exciting developments and has provided answers to a series of practical questions such as: what is the maximum flow per unit time from source to sink in a network of pipes or how to color the regions of a map using the minimum number of colours so that neighbouring regions receive different colors among other important problems. In addition to the developments in mathematical graph theory, the study of networks has seen important achievements in some specialized contexts, as for instance in the social sciences. Most of the results of graph theory relevant to large complex networks, are related to the simplest models of random graphs.

\vspace{0.3cm}
\noindent Recent years however have witnessed a substantial new movement in network research, with the focus shifting away from the analysis of single small graphs and the properties of individual vertices or edges to considerations of ``large scale" statistical properties. The great majority of real world networks, including the World Wide Web, the Internet, basic cellular networks, social networks and many others have a more complex architecture than classical random graphs. Abstracting the network details away allows one to concentrate on the phenomena intrinsically connected with the underlying geometry, and discover connections between the metric properties and the network characteristics. Over the past few years, there has been growing evidence that many communication networks have characteristics of negatively curved spaces \cite{H0, H1, H2, H3, H4}. From the large scale point of view, it has been experimentally observed that, on the Internet and other networks, traffic seems to concentrate quite heavily on some very small subsets.

\vspace{0.3cm}
\noindent We believe that many of the complex real world networks have characteristics of negatively curved or more generally Gromov's hyperbolic spaces. In Figure 1, we observe a picture of the World Wide Web and the Internet network that suggests a hyperbolic structure. In this project we continue the analysis and approach done in \cite{Bar-Tucci}. We study the traffic behaviour for large Gromov's hyperbolic spaces when the traffic rate decays exponentially with the metric distance between the nodes. We prove that under general conditions there exists a phase transition between local and global traffic. More specifically, assume that the traffic rate between two nodes $u$ and $v$ in our network is given by $R(u,v)=\beta^{-d(u,v)}$ where $d(u,v)$ is the distance between the nodes. We show that there exists a constant $\beta_c$ such that if $1<\beta<\beta_c$ the traffic is global and there is a small set of highly congested nodes called the core. However, if $\beta>\beta_c$ then the traffic is essentially local and the core is empty. This implies in particular, that polynomially decaying rate functions do not affect the locality of the traffic, and the existence or non--existence of a core. The dichotomy of global versus local traffic is more important than ever. A recent study showed that the consumer broadband usage and global IP network traffic continues to climb at an overwhelming pace due to new forms and expanded usage of interactive media, and the explosion of video content across multiple devices. The study projects that global IP traffic will increase fivefold by 2013. The major growth driver is video and it is expected that by 2013 the sum of all forms of video (TV, VoD, Internet video and P2P) will exceed 90\% of the total consumer IP traffic. This, if not handled appropriate, will generate huge congestions in our networks. Our results in particular imply that if we foment and incentive local traffic instead of global traffic this problem can be minimized. 

 \begin{figure}[!Ht]\label{internet}
  \begin{center}
    \includegraphics[width=1.8in]{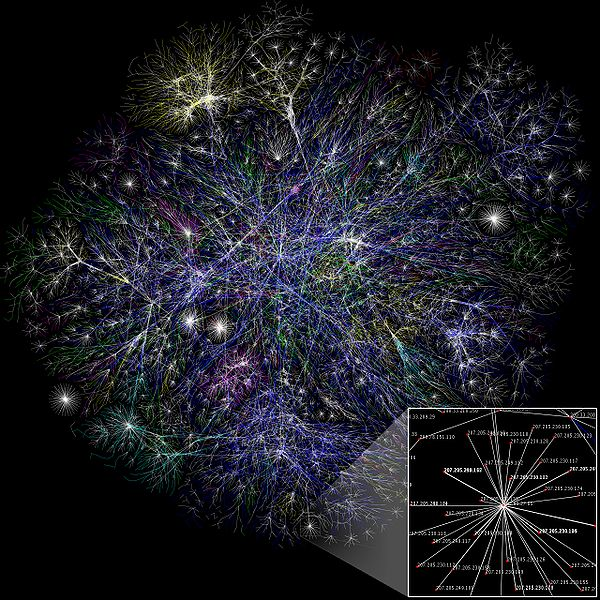}
    \includegraphics[width=2.4in]{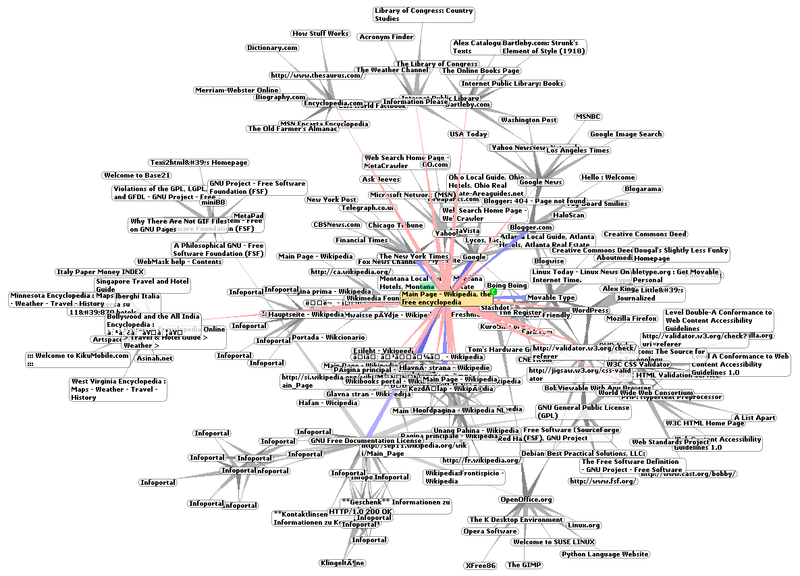}
    \caption{On the right we see a visualization of the various routes of the Internet. On the left we see a map of the World Wide Web.}
  \end{center}
\end{figure}

\vspace{0.3cm}
\noindent In Section \ref{prelim}, we review the concept of Gromov's hyperbolic space and present some of the important examples and properties. We also recall the construction of the boundary of an hyperbolic space and its visual metric. In Section \ref{tree}, we study the traffic phenomena in a general locally finite tree when the rate decays exponentially with the distance. Finally, in Section \ref{hyper} we study the asymptotic traffic behaviour in general Gromov's hyperbolic graphs when the traffic decays exponentially with the distance and prove our main results.

\vspace{0.3cm}
\noindent
{\it Acknowledgement:} We would like to thank Iraj Saniiee for many helpful discussions and comments. This work was supported by AFOSR Grant No. FA9550-08-1-0064.

\section{Preliminaries}\label{prelim}
\noindent In this Section we review the notion of Gromov's $\delta$--hyperbolic space as well as some of the basic properties, theorems and constructions.

\subsection{$\delta$--Hyperbolic Spaces}
\noindent There are many equivalent definitions of Gromov's hyperbolicity but the one we take as our definition is the property that triangles are {\it slim}.

\begin{df}
Let $\delta>0$. A geodesic triangle in a metric space $X$ is said to be $\delta$--slim if each of its sides is contained in the $\delta$--neighbourhood of the union of the other two sides. A geodesic space $X$ is said to be $\delta$--hyperbolic if every triangle in $X$ is $\delta$--slim. 
\end{df}

\noindent It is easy to see that any tree is $0$-hyperbolic. Other examples of hyperbolic spaces include, any finite graph, the fundamental group of a surface of genus greater or equal than 2, the classical hyperbolic space, and any regular tessellation of the hyperbolic space (i.e. infinite planar graphs with uniform degree $q$ and $p$--gons as faces with $(p-2)(q-2)>4$).

\begin{df} (Hyperbolic Group) A finitely generated group $\Gamma$ is said to be word--hyperbolic if there is a finite generating set $S$ such that the Cayley graph $C(\Gamma,S)$ is $\delta$--hyperbolic with respect to the word metric for some $\delta$.
\end{df}

\noindent It turns out that if $\Gamma$ is a word hyperbolic group then for any finite generating set $S$ of $\Gamma$ the corresponding Cayley graph is hyperbolic, although the hyperbolicity constant depends on the choice of $S$.

\begin{df}(Gromov's Product) Let $(X,d)$ be a metric space. For $x,y$ and $z\in X$ we define 
$$(y,z)_{x}:=\frac{1}{2}(d(x,y)+d(x,z)-d(y,z)).$$

\noindent We call $(y,z)_{x}$ the Gromov's product of $y$ and $z$ with respect to $x$. 
\end{df}
\noindent In hyperbolic metric spaces the Gromov's product measures how long two geodesics travel close together. Namely if $x,y$ and $z$ are three points in a $\delta$ hyperbolic metric space $(X,d)$, then the initial segments of length $(y,z)_{x}$ of any two geodesics $[x,y]$ and $[x,z]$ are $2\delta$ Hausdorff close. Moreover, in the case of Gromov's product $(y,z)_{x}$ approximates within $2\delta$ the distance from $x$ to a geodesic $[y,z]$.

\subsection{Boundary of Hyperbolic Spaces}

\vspace{0.3cm}
\noindent We say that two geodesic rays $\gamma_{1}:[0,\infty)\to X$ and $\gamma_{2}:[0,\infty)\to X$ are equivalent and write $\gamma_{1}\sim\gamma_{2}$ if there is $K>0$ such that for any $t\geq 0$
$$
d(\gamma_{1}(t),\gamma_{2}(t))\leq K.
$$
It is easy to see that $\sim$ is indeed an equivalence relation on the set of geodesic rays. Moreover, two geodesic rays $\gamma_{1},\gamma_{2}$ are equivalent if and only if their images have finite Hausdorff distance. The Hausdorff distance is defined as the infimum of all the numbers $H$ such that the images of $\gamma_{1}$ is contained in the $H$--neighbourhood of the image of $\gamma_{2}$ and vice versa.

\vspace{0.3cm}
\noindent The boundary is usually defined as the set of equivalence classes of geodesic rays starting at the base--point, equipped with the compact--open topology. That is to say, two rays are ``close at infinity'' if they stay close for a long time. We make this notion precise.

\begin{df}(Geodesic Boundary)
Let $(X,d)$ be a $\delta$--hyperbolic metric space and let $x_0\in X$ be a base--point. We define the relative geodesic boundary of $X$ with respect to the base--point $x_0$ as 

\begin{equation}
 \partial {X}:=\{[\gamma]\,\,:\,\gamma:[0,\infty)\to X \,\,\text{is a geodesic ray with} \,\,\gamma(0)=x_0\}.
\end{equation}
\end{df}

\vspace{0.3cm}
\noindent It turns out that the boundary has a natural metric.

\begin{df}
Let $(X,d)$ be a $\delta$--hyperbolic metric space. Let $a>1$ and let $x_{0}\in X$ be a base--point. We say that a metric $d_{a}$ on $\partial X$ is a visual metric with respect to the base point $x_{0}$ and the visual parameter $a$ if there is a constant $C>0$ such that the following holds:

\begin{enumerate}
\item The metric $d_{a}$ induces the canonical boundary topology on $\partial X$.
\item For any two distinct points $p,q\in\partial X$, for any bi-infinite geodesic $\gamma$ connecting $p,q$ in $X$ and any $y\in\gamma$ with  $d(x_{0},\gamma)=d(x_{0},y)$ we have:
$$\frac{1}{C}a^{-d(x_{0},y)}\leq d_{a}(p,q)\leq Ca^{-d(x_{0},y)}.$$
\end{enumerate}
\end{df}

\begin{teo}(\cite{Coor}, \cite{Harper})\label{visual_teo}
Let $(X,d)$ be a $\delta$--hyperbolic metric space. Then:
\begin{enumerate}
\item There is $a_{0}>1$ such that for any base point $x_{0}\in X$ and any $a\in (1,a_{0})$ the boundary $\partial X$ admits a visual metric $d_{a}$ with respect to $x_{0}$.
\item Suppose $d'$ and $d''$ are visual metrics on $\partial X$ with respect to the same visual parameter $a$ and the base points $x_{0}'$ and $x_{0}''$ accordingly. Then $d'$ and $d''$ are Lipschitz equivalent, that is there is $L>0$ such that
$$d'(p,q)/L\leq d''(p,q)\leq Ld'(p,q)\hspace{0.5cm}\text{for any}\,\,p,q\in\partial X.$$ 
\end{enumerate}
\end{teo}

\vspace{0.2cm}
\noindent The metric on the boundary is particularly easy to understand when $(X,d)$ is a tree. In this case $\partial X$ is the space of ends of $X$. The parameter $a_{0}$ from the above proposition is $a_{0}=\infty$ here and for some base point $x_{0}\in X$ and $a>1$ the visual metric $d_{a}$ can be given by an explicit formula:
$$
d_{a}(p,q)=a^{-d(x_{0},y)}
$$
for any $p,q\in\partial X$ where $[x_{0},y]=[x_{0},p)\cap [x_{0},q)$ so that $y$ is the bifurcation point for the geodesic rays $[x_{0},p)$ and $[x_{0},q)$.

\vspace{0.2cm}
\noindent Here are some more examples of boundaries of hyperbolic spaces (for more on this topic see \cite{Coor, Harper, Gromov}.)

\begin{ej}
\begin{enumerate}
\item If $X$ is a finite graph then $\partial X =\emptyset$.
\item If $X =\Z$, the infinite cyclic group, then $\partial X $ is homeomorphic to the set $\{0,1\}$ with the discrete topology.
\item If $n\geq 2$ and $X =\F_{n}$, the free group of rank $n$, then $\partial X $ is homeomorphic to the space of ends of a regular $2n$--valent tree, that is to a Cantor set.
\item Let $S_{g}$ be a closed oriented surface of genus $g\geq 2$ and let $X =\pi_{1}(S_{g})$. Then $X $ acts geometrically on the hyperbolic plane $\mh^{2}$ and therefore the boundary is homeomorphic to the circle $S^{1}$. 
\item Let $M$ be a closed $n$--dimensional Riemannian manifold of constant negative sectional curvature and let $X =\pi_{1}(M)$. Then $X$ is word hyperbolic and $\partial X $ is homeomorphic to the sphere $S^{n-1}$.
\item The boundary of the classical $n$ dimensional hyperbolic space $\mathbb{H}^n$ is $S^{n-1}$.
\end{enumerate}
\end{ej}

\subsubsection{Hausdorff dimension and Growth Function}

\noindent Given an hyperbolic graph $X$, it is natural to ask about the Hausdorff dimension of its boundary set. Let $(X,d)$ be a complete metric space. One defines the $\alpha$--Hausdorff measure of a set $Z\subset X$ as 
$$
m_{H}(Z,\alpha):=\liminf_{\epsilon\to 0}{\sum_{U\in\mathcal{G}_{\epsilon}}{(\mathrm{diam}(U))^{\alpha}}},
$$
the infimum being taken over all the covers $\mathcal{G}_{\epsilon}$ of $Z$ by open sets of diameter at most $\epsilon$. The usual Hausdorff dimension of $Z$ is taken
$$
\dim_{H}(Z)=\inf\{\alpha\,:\,m_{H}(Z,\alpha)=0\}=\sup\{\alpha\,:\,m_{H}(Z,\alpha)=+\infty\}.
$$
When $m_{H}(X,\dim_{H}(X))$ is finite and non zero, the function $Z\to m_{H}(Z,\dim_{H}(X))$ is after normalization a probability measure on $X$, called the Hausdorff measure. The critical exponent of base $a$ of an infinite graph is defined as
\begin{equation}\label{growth}
 e_{a}(X):=\limsup_{R\to\infty}{\frac{\log_{a}(|\{x \in X\,:\,d(x_0,x)\leq R\}|)}{R}}.
\end{equation}
It is known that the Hausdorff dimension of the boundary $\partial X$ with respect to the visual metric $d_a$ is equal to $e_{a}(X)$ (see for instance \cite{Coor, Coorna, Bla}).

%\begin{teo}\label{fun}(See \cite{Bla})
%Let $(X,d)$ be a non--elementary hyperbolic graph acting and et $d_{a}$ be a visual metric on the boundary $\partial X$, and let $B(x,r)$ be the ball of center $x\in\partial X$ and radius $r$ for %the distance $d_{a}$. Let $\nu$ be the harmonic measure of a random walk $(Z_{n})_{n}$ whose increments are given by a symmetric law $\mu$ with finite first moment. Then the point--wise Hausdorff %dimension 
%$$
%\lim_{r\to 0}{\frac{\log\nu(B(x,r))}{\log r}}
%$$
%exists for $\nu$--almost every $x\in\partial X$, and is independent from the choice of $x$. Moreover, for $\nu$--almost every $x\in\partial X$.
%\begin{equation}
% 0<\dim_{H}(\nu)=\lim_{r\to 0}{\frac{\log\nu(B(x,r))}{\log r}}=\frac{h}{\log(a)\,l}
%\end{equation}
%where $l>0$ denotes the rate of escape of the walk wrt $d$ and $h$ the asymptotic entropy of the walk.
%\end{teo}

%\begin{obs}
%The support of the harmonic measure is $\partial X$. The Hausdorff dimension $\partial X$ is equal to $e_{a}(\Gamma)$ (see \cite{Coorna}). We have a fundamental relation between $h, l$ and $v$ (see %\cite{Bla}): 
%\begin{equation}
%h\leq \log(a)e_{a}(\Gamma)l
%\end{equation}
%\noindent In view of this inequality and Theorem \ref{fun} we see that 
%\begin{equation}
% 0<\dim_{H}(\nu)\leq e_{a}(\Gamma).
%\end{equation}
%\end{obs}

\section{Asymptotic Traffic Flow in a Tree}\label{tree}

\noindent In this Section we study the asymptotic traffic behaviour in a locally finite tree when the traffic decays exponentially with the distance. More specifically, let $\{k_{l}\}_{l=0}^{\infty}$ be a sequence of positive integers with $k_{0}=1$. For each sequence like this we consider the infinite tree $T$ with the property that each element at depth $l$ has $k_{l+1}$ descendants. In other words, the root has $k_{1}$ descendants, each node in the first generation has $k_{2}$ descendants and so on. The root is considered the 0 generation. Let us denote by $T_{n}$ the finite tree generated by the first $n$ generations of $T$. Let $N=N(n)$ be the number of elements in $T_{n}$. It is clear that
$$
N(n)=1+k_{1}+k_{1}k_{2}+\ldots+k_{1}k_{2}\ldots k_{n}=\sum_{l=0}^{n}{\prod_{i=0}^{l}{k_{i}}}.
$$
For each fixed $n\geq 1$, assume that there is traffic between $\partial T_{n}$, the leaves of the truncated tree $T_{n}$. We also assume also that the traffic rate between $x_i$ and $x_j$ in $\partial T_{n}$ depends only on the distance between these two leaves and decays exponentially. More specifically,
$$
R(i,j)=\beta^{-d(x_i,x_j)}\hspace{0.5cm} \text{where}\,\,\, \beta>1.
$$

\vspace{0.3cm}
\noindent Denote by $x_{0}$ the root of the tree. For simplicity let us first assume that the tree $T$ is $(k+1)$--regular which is equivalent to assume that $k_{l}=k$ for all $l\geq 1$. It is an easy observation to see that the number of elements of $\partial T_{n}$ is equal to $N(n):=|\partial T_{n}|=(k+1)k^{n-1}$. Let us denote these points as $x_{1},\ldots,x_{N}$. Let $n_{p}=|\{x_{i}\,\,:\,\,d(x_{1},x_{i})=p\}|$. Then
\begin{equation}
n_{p} = \left\{
\begin{array}{ll}
(k-1)k^{r-1} & \text{if }  p=2r \,\,\,\text{for}\,\,\, 0\leq r\leq n\\
0 & \text{otherwise }  \\
\end{array} \right.
\end{equation}
The total traffic between the points $x_{1},x_{2},\ldots,x_{N}$ is 
\begin{eqnarray*}
T(n) & = & N\cdot\Bigg(\sum_{p=0}^{\infty}{n_{p}\,\beta^{-p}}\Bigg)=N\cdot\Big(1+(k-1)\cdot\sum_{i=0}^{n-1}{k^{i}\,\beta^{-2(i+1)}}\Big) \\
& = & N\cdot\Bigg(1+\frac{k-1}{\beta^2}\cdot\frac{(k/\beta^{2})^{n}-1}{(k/\beta^{2})-1}\Bigg)
\end{eqnarray*}
The total traffic passing through the root of the tree is $N(k-1)k^{n-1}\beta^{-2n}$. Hence the proportion of the traffic passing through the root of the tree is equal to 
$$
P(n)=\frac{(k-1)k^{n-1}\beta^{-2n}}{1+\beta^{-2}(k-1)\frac{(k\beta^{-2})^{n}-1}{(k\beta^{-2})-1}}.
$$
Here we can distinguish two cases. The first case is $\beta\geq\sqrt{k}$. In this case 
$$
\lim_{n \to\infty}{P(n)}=0.
$$ 
The other case is $1<\beta<\sqrt{k}$ in which 
$$
\lim_{n\to\infty}{P(n)}=1-\frac{\beta^{2}}{k}.
$$
This shows in particular that if the traffic decay is sub--exponential then the asymptotic proportion of the traffic through the root is $1-\frac{1}{k}$. A similar analysis and conclusion can be carried out for the general tree $T$ as long as there is an upper bound on the coefficients $k_{l}$. We deduce a more general theorem in the next Section which includes this result as a particular case.

\section{Asymptotic Traffic Flow in a $\delta$--Hyperbolic Graph}\label{hyper}

\noindent In this Section, we study the asymptotic traffic flow in a $\delta$--hyperbolic graph.  Throughout this Section we assume that $X$ is an infinite, locally finite (every node has finite degree), simple (no loops or multiple edges) graph. Assume that there exists $\delta>0$ such that $X$ is Gromov's $\delta$--hyperbolic. Let $x_{0}\in X$ be a fixed base point and let 
$$ 
\{x_{0}\}=X_{0}\subset X_{1}\subset X_{2}\subset \ldots\subset X_{n}\subset\ldots\subset X
$$
be a sequence of finite subsets with the properties that: 
\begin{itemize}
\item $\cup_{n\geq 1}{X_{n}}=X$,
\item for every $x\in X_{n}$ and for every geodesic segment $[0,x]$ connecting $0$ and $x$  then every intermediate point belongs to $X_{n}$. 
\end{itemize}

\vspace{0.2cm}
\noindent Denote as usual by $\partial X_{n}$ the boundary set of $X_{n}$ in $X$ and recall that a point $y$ belongs to $\partial X_{n}$ if $y\in X_{n}$ and there exists $z\in X\setminus X_{n}$ such that $z\sim y$ ($z$ and $y$ are adjacent).

\begin{figure}[!Ht]
  \begin{center}
    \includegraphics[width=5cm]{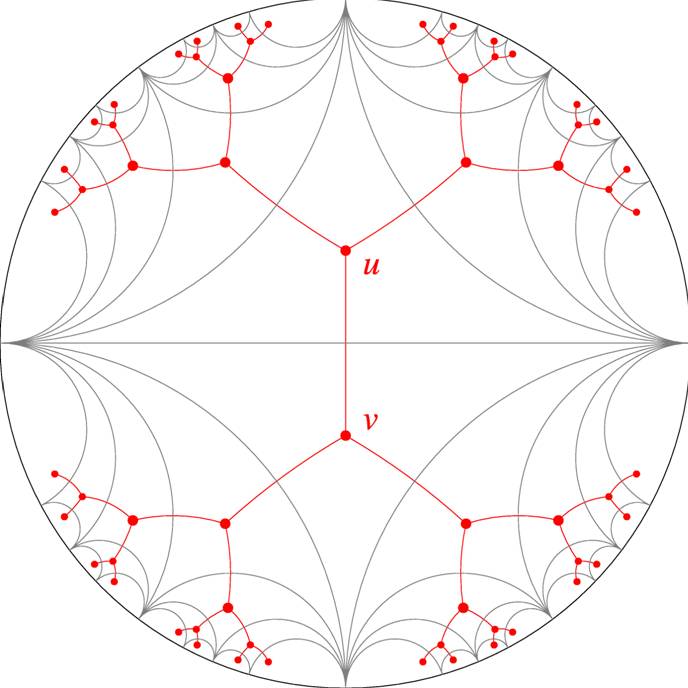}
    \caption{Regular tree embedded in $\B^{2}$.}
  \end{center}
\end{figure}

\vspace{0.3cm}
\noindent We assume that for each fixed $n$ the traffic flow goes from nodes in $\partial X_{n}$ to nodes in $\partial X_{n}$. We also assume that there is a non--increasing, and continuous function $f:[0,\infty)\to [0,\infty)$ such that the traffic rate between $x$ and $y$ in $\partial X_{n}$ is equal to 
\begin{equation}
R(x,y)=f(d(x,y))
\end{equation} 
where $d(x,y)$ is the distance between these two points. The traffic flow goes through the geodesic connecting $x$ and $y$, and if there are more than one geodesic connecting these points we assume that the load is divided equally between the different paths. We pay special attention to the case where $f(t)=\beta^{-t}$ for $\beta>1$.

\vspace{0.3cm}
\noindent Of central importance in this work is the case where the sets $\{X_{n}\}_{n}$ are balls. More precisely, assume that 
\begin{equation}\label{balls}
X_{n}:=\{x\in X\,\,:\,\,d(x_{0},x)\leq n\}.
\end{equation}
In this case it is clear that $\partial X_{n}=\{x\in X\,\,:\,\,d(x_{0},x)=n\}$. Recall from Section \ref{prelim} that for any $x$ and $y$ in $X$,
\begin{equation}\label{xy}
(x,y)_{x_{0}}=\frac{1}{2}(d(x_{0},x)+d(x_{0},y)-d(x,y)),
\end{equation} 
and 
$$
h(x,y)=d(x_{0},\gamma_{x,y}),
$$ 
where $\gamma_{x,y}$ is the geodesic connecting $x$ and $y$ (if there is more than one geodesic connecting $x$ and $y$ then we consider the minimum).

\vspace{0.3cm}
\noindent It is not difficult to see that (see \cite{Gromov})
$$
h(x,y)-4\delta\leq (x,y)_{x_{0}}\leq h(x,y).
$$
\noindent In particular, using equation (\ref{xy}) we see that for every pair of points $x$ and $y$ in $\partial X_{n}$ 
\begin{equation}\label{eqq}
n-\frac{d(x,y)}{2} \leq h(x,y)\leq n-\frac{d(x,y)}{2}+4\delta.
\end{equation}

\vspace{0.3cm}
\noindent By theorem \ref{visual_teo} we know that exists $a_{0}>1$ such that for all $a\in (1,a_{0})$ the boundary $\partial X$ admits a visual metric $d_{a}$ with base point $x_{0}$. Hence, there exists $C>0$ such that:

\begin{enumerate}
 \item The metric $d_{a}$ induces the canonical boundary topology on $\partial X$.
 \item For all $p\neq q\in\partial X$  
\begin{equation}\label{da}
\frac{1}{C}\,a^{-h(p,q)}\leq d_{a}(p,q)\leq C\,a^{-h(p,q)}.
\end{equation}
\end{enumerate}

\noindent Let $r>0$ be fixed, and let $p$ and $q\in\partial X$ with $h(p,q)\leq r$ then $\frac{1}{C}\,a^{-r}\leq \frac{1}{C}\,a^{-h(p,q)}\leq d_{a}(p,q)$. \noindent Therefore,
\begin{equation}
d_{a}(p,q)\geq \frac{1}{C}\,a^{-r}.
\end{equation}
On the other hand, if $d_{a}(p,q)\geq C\,a^{-r}$ then $h(p,q)\leq r$. Note that $C$ is a positive fixed constant that only depends on $a$.

\vspace{0.3cm}
\noindent Let $p$ and $q\in\partial X$ and $x_{n}$ and $y_{n}\in\partial X_{n}$ such that $x_{n}\to p$ and $y_{n}\to q$ as $n$ goes to infinity. Then by equation (\ref{eqq}) 
$$
n-\frac{d(x_n,y_n)}{2}\leq h(x_n,y_n)\leq n-\frac{d(x_n,y_n)}{2}+4\delta.
$$ 
Since $\lim_{n\to\infty}{h(x_{n},y_{n})}=h(p,q)$ we conclude that for $n$ sufficiently large 
\begin{equation}\label{eq_h}
2(n-h(p,q))\leq d(x_{n},y_{n})\leq 2(n-h(p,q))+4\delta.
\end{equation}

%\vspace{0.3cm}
%\noindent Consider the random walk $\{Z_{n}\}_{n=0}^{\infty}$ such that $Z_{0}=x_{0}$ and at every stage all the adjacent nodes are equally likely. In other words, 
%\begin{equation}
%\mathbb{P} (Z_{n+1}=y|Z_{n}=x) = \left\{
%\begin{array}{ll}
%\frac{1}{\mathrm{deg}(x)} & \text{if }  y\sim x\\
%0 & \text{otherwise } 
%\end{array} \right.
%\end{equation}
%We know that almost surely the random walk converges to a point in the boundary $\lim_{n\to\infty}{Z_{n}}=Z_{\infty}\in\partial X$, and the law of $Z_{\infty}$ defines an harmonic measure $\nu$ on %$\partial X$. We assume that 
%\begin{equation}\label{nu_dim}
%0<\dim_{H}(\nu)=\lim_{r\to 0}{\frac{\log\nu(B(p,r))}{\log r}}
%\end{equation}
%for $\nu$--almost every $p$ in $\partial X$. In the case where $X$ is a $\delta$--hyperbolic group this is true by Theorem \ref{fun}, and moreover this measure satisfies 
%$$
%0<\dim_{H}(\nu)=\lim_{r\to 0}{\frac{\log\nu(B(p,r))}{\log r}}=\frac{h}{\log(a)l}\leq e_{a}(\Gamma)
%$$
%\noindent for $\nu$--almost every $p\in\partial \Gamma$ where $h$ and $l$ are the entropy and drift of the random walk, and $\delta(\Gamma)$ is the growth exponent of the group. For every integer %$n\geq 1$, and every $x\in\partial X_{n}$ define
%$$
%L_{x}=|\{\text{geodesics connecting $x_{0}$ and $x$}\}|,
%$$
%and  
%$$
%R_{n}=|\{\text{geodesics starting at $x_{0}$ of length $n$}\}|.
%$$ 
\noindent Let $\mu_{n}$ be the uniform measure in $\partial X_{n}$ defined as
\begin{equation}\label{visual}
\mu_{n}= \sum_{x\in\partial X_{n}}{\delta_{x}}.
\end{equation}
This measure defines a visual Borel probability measure $\mu_{n}^{v}$ in the boundary $\partial X$. The way this measure is defined is described below. 

\begin{df}
Let $A\subseteq\partial X$ be a Borel subset. For each $a\in A$ consider the set of sequences $\{x_{k}\}_{k=0}^{\infty}$ such that:
$x_{0}=0$, the sequence is a geodesic ray in $X$ that converges to $a$. These sequences correspond to rays connecting $0$ with $a$. Let $C_{A}$ be the set of points in $X$ that belong to some ray connecting $0$ with $a$ for some $a\in A$. We define 
\begin{equation}
\mu_{n}^{v}(A):=\frac{\mu_{n}{(X_{n}\cap C_{A})}}{\mu_n(X_n)}.
\end{equation}
\end{df}

\noindent It can be shown that these visual measures converge weakly to a conformal measure $\nu$ in $\partial X$ (see \cite{Coorna} for more details on this and the construction of the conformal measures)
\begin{equation}\label{weak}
\mu_{n}^{v}\to\nu\hspace{0.5cm}\text{weakly}.
\end{equation}
Moreover, see \cite{Coorna} the Hausdorff dimension of this measure with respect to the visual metric $d_a$ is equal to $e_{a}(X)$ as in equation (\ref{growth}). Moreover, it was proved in Proposition 7.4 of \cite{Coorna} that there exists a constant $K>1$ such that
\begin{equation}\label{hauss}
K^{-1}r^{D}\leq \nu(B(x,r))\leq Kr^{D}
\end{equation}
for every point $x\in \partial X$ and every $r\geq 0$ with respect to the metric $d_a$.

\vspace{0.3cm}
\noindent We assume, as we mentioned before, that for each fixed $n$, the traffic flow between two points $x$ and $y$ in $\partial X_{n}$ is equal to is $R(x,y)=f(d(x,y))$ for some fixed function $f$. The total traffic passing through the network $X_{n}$ is equal to 
\begin{equation}\label{tt}
 T(n)=\int_{\partial X_{n}\times\partial X_{n}}{R(x,y)\,d\mu_{n}(x)\,d\mu_{n}(y)}.
\end{equation}
Let $r\geq 0$, and denote by $T_{r}(n)$ the total traffic passing through $B(x_{0},r)$. Then 
\begin{equation}\label{tr}
 T_{r}(n)=\int_{\partial X_{n}}  {\Bigg(\int_{E_{x}^{r}}{R(x,y)\,d\mu_{n}(y)}\Bigg)\,d\mu_{n}(x)}
\end{equation}
where $E_{x}^{r}=\{y\in\partial X_{n}\,\,:\,\,h(x,y)\leq r\}$ and $x\in\partial X_{n}$.

\subsection{Exponential Decay}
In what follows we assume that the traffic rate decays exponentially with the distance, i.e. there exist $\beta>1$ such that
$$
R(x,y)=\beta^{-d(x,y)}.
$$
Now we are ready to state our main theorem.

\begin{teo}\label{main}
Let $X$ be an infinite $\delta$--hyperbolic graph and let $\{X_{n}\}_{n=0}^{\infty}$ be as in (\ref{balls}). Then there exist a constant $\beta_c$ such that if $1<\beta<\beta_c$ then for every $\epsilon>0$ there exist $r_{0}>0$ such that for all $r\geq r_{0}$ 
\begin{equation}\label{small_beta}
\lim_{n\to\infty}{\frac{T_{r}(n)}{T(n)}}\geq 1-\epsilon.
\end{equation}
Moreover, if $\beta > \beta_c$ then for every $r>0$ 
\begin{equation}\label{big_beta}
\lim_{n\to\infty}{\frac{T_{r}(n)}{T(n)}}=0.
\end{equation}
Moreover, $\beta_c=e^{e(X)/2}$ where $e(X)$ is defined as $e_a(X)$ with respect to the natural logarithm.
\end{teo}

\begin{proof}
Using Equations (\ref{eq_h}), (\ref{weak}), (\ref{tt}) and (\ref{tr}) we see that  
%\begin{equation}
%T(n)\approx \int_{\partial X\times\partial X}{\beta^{-2n}\beta^{2h(x,y)}\,d\nu(x)\,d\nu(y)},
%\end{equation}
%and 
%\begin{equation}
%T_{r}(n)\approx \int_{\partial X}{\Big(\int_{E_{x}^{r}}{\beta^{-2n}\beta^{2h(x,y)}\,d\nu(y)}\Big)\,d\nu(x)}.
%\end{equation}
%Define $\lambda(r)$ as the asymptotic proportion of traffic passing through $B(x_{0},r)$, then 
$$
\lim_{n\to\infty}{\,\,\frac{T_{r}(n)}{T(n)}}=\frac{\int_{\partial X}{\Big(\int_{E_{x}^{r}}{\beta^{2h(x,y)}\,d\nu(y)}\Big)\,d\nu(x)}}{\int_{\partial X\times\partial X}{\beta^{2h(x,y)}\,d\nu(x)\,d\nu(y)}},
$$
where $E_{x}^{r}=\{y\in\partial X\,\,:\,\,h(x,y)\leq r\}$. Let $F:\partial X\times\R^{+}\to [0,1]$ be the function defined by 
\begin{equation}
F(x,r):= \frac{\int_{E_{x}^{r}}{\beta^{2h(x,y)}\,d\nu(y)}}{\int_{\partial X}{\beta^{2h(x,y)}\,d\nu(y)}}.
\end{equation}
By the compactness of the boundary $\partial X$, to prove that for $1<\beta<\beta_c$ equation (\ref{small_beta}) holds it is enough to prove that 
$$
\lim_{r\to\infty}{F(x,r)}=1\quad\quad\text{for almost every $x\in\partial X$.}
$$ 
Analogously, if $\beta > \beta_c$ to prove that (\ref{big_beta}) holds it is enough to prove that 
$$
\lim_{r\to\infty}{F(x,r)}=0\quad\quad\text{for almost every $x\in\partial X$.}
$$ 
\noindent Since the function $h:\partial X\times\partial X\to \R^{+}$ only takes integer values (recall that $X$ is a graph) then 
\begin{equation}
\int_{E_{x}^{r}}{\beta^{2h(x,y)}\,d\nu(y)}=\sum_{k=0}^{r}{\beta^{2k}\cdot\nu\big(\{y\in\partial X\,:\,h(x,y)=k\}\big)},
\end{equation}
and 
\begin{equation}\label{eq_series}
\int_{\,\partial X\,\,}{\beta^{2h(x,y)}\,d\nu(y)}=\sum_{k=0}^{+\infty}{\beta^{2k}\cdot\nu\big(\{y\in\partial X\,:\,h(x,y)=k\}\big)}.
\end{equation}
Hence, $\lim_{r\to\infty}{F(x,r)}=1$ if the series in equation (\ref{eq_series}) converges for almost every $x\in\partial X$, and $\lim_{r\to\infty}{F(x,r)}=0$ if the series diverges. Note that by equation (\ref{da}) we have that 
$$
\nu\big(\{y\in\partial X\,:\,h(x,y)=k\}\big)\leq \nu\big(\{y\in\partial X\,:\,d_{a}(x,y)\leq Ca^{-k}\}\big).
$$
Using equation (\ref{hauss}) we know that there exists a constant $C$ such that for $\nu$--almost every $x$ in the boundary 
$$
\nu\big(\{y\in\partial X\,:\,d_{a}(x,y)\leq L\}\big)\leq KL^{e_{a}(X)},
$$
for every $L>0$. Therefore,
$$
\sum_{k=0}^{+\infty}{\beta^{2k}\cdot\nu\big(\{y\in\partial X\,:\,h(x,y)=k\}\big)}\leq\sum_{k=0}^{+\infty}{\beta^{2k}K(Ca^{-k})^{e_{a}(X)}}.
$$
Since 
\begin{equation}
\sum_{k=0}^{+\infty}{\beta^{2k}K(Ca^{-k})^{e_{a}(X)}}=KC^{e_{a}(X)}\cdot\sum_{k=0}^{+\infty}{\Bigg(\frac{\beta^{2}}{a^{e_{a}(X)}}\Bigg)^{k}}
\end{equation}
\noindent this series converges if and only if $\beta< a^{e_{a}(X)/2}=e^{e(X)/2}$. Analogously, if $\beta>e^{e(X)/2}$ then the series diverges and the traffic is asymptotically local.
\end{proof}

\begin{figure}[!Ht]
  \begin{center}
    \includegraphics[width=5cm]{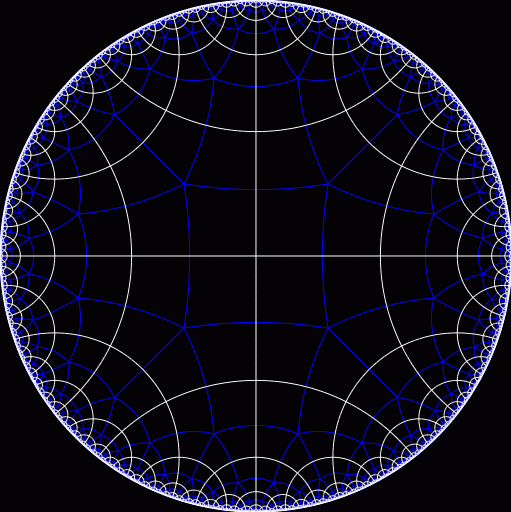}
    \caption{$H_{5,4}$ and $H_{4,5}$ tessellations of the Poincar\'e disk.}
  \end{center}
\end{figure}

\begin{obs}
For the case of the $(k+1)$--regular tree it is true that 
$$e_{a}(\Gamma)=\limsup_{n\to\infty}{\frac{\log_{a}(|\{\gamma\in\Gamma\,:\,|\gamma|\leq n\}|)}{n}}=\log_{a}(k).$$

\noindent Hence for the $(k+1)$--regular tree 
$$
1<\beta< \beta_c = e^{\frac{\log(k)}{2}}=\sqrt{k}.
$$ 
Note that in Section \ref{tree} we proved this result directly.
\end{obs}

\begin{ej}
Let $X=S_{g}$ be the Cayley graph of surface group of genus $g\geq 2$ or more generally, let $X=H_{p,q}$ be the regular hyperbolic tessellations of the hyperbolic space (i.e. infinite planar graphs with uniform degree $q$ and $p$--gons as faces with $(p-2)(q-2)>4$). The rate of growth of these tessellations has been extensively studied and the numbers $e(X)$ are known as the roots of certain polynomials called Salem polynomials (see \cite{bart} for more details). In particular, it can be shown that for the $(5,4)$ tessellation 
$$
e(X)=\log\Big(\frac{3+\sqrt{5}}{2}\Big)
$$
and hence $\beta_c=\sqrt{\frac{3+\sqrt{5}}{2}}\approx 1.61803$.
\end{ej}


\begin{thebibliography}{19}


\bibitem{Bar-Tucci} Y. Baryshnikov and G. Tucci, {\it Asymptotic traffic flow in an Hyperbolic Network I : Definition and Properties of the Core}, preprint.

\bibitem{bart} L. Barthold and T. Ceccherini-Silberstein, {\it Salem Numbers and Growth Series of Some Hyperbolic Graphs}, Geometriae Dedicata
Vol. 90, No. 1, pp. 107-114, 2002.

\bibitem{Bla} S. Blach\`{e}re, P. Ha\"{\i}ssinsky and P. Mathieu, {\it Harmonic measures versus quasi-conformal measures for hyperbolic groups}, arXiv:0806.3915v1.

\bibitem{Bonk} M. Bonk and B. Kleiner, {\it Rigidity for quasi--Mobius group actions}, J. Differential Geom. 61, no. 1, 81--106, 2002.

\bibitem{internet} S. Carmi, S. Havlin, S. Kirkpatrick, Y. Shavitt and E. Shir, {\it A model of Internet topology using $k$-shell decomposition}, PNAS 104, no. 27, 2007.

\bibitem{Coorna} M. Coornaert, {\it Measures de Patterson--Sullivan sur le bord d'un espace hyperbolique au sens de Gromov}, Pacific J. Math. 159, no. 2, 241--270, 1993.

\bibitem{Coor} M. Coornaert, T. Delzant and A. Papadopoulos, {\it Geometrie et theorie des groupes}, Springer--Verlag, Berlin, 1993.

\bibitem{Harper} E. Ghys and P. de la Harper, {\it Sur les groupes hyperboliques d'apres Mikhael Gromov}, Birkhauser Boston Inc., Boston, MA, 1990.

\bibitem{Gromov} M. Gromov, {\it Hyperbolic groups}, Essays in group theory, Springer, New York, 1987, pp. 75-263.

\bibitem{H1} E. Jonckheere, P. Lohsoonthorn and F. Bonahon, {\it Scaled Gromov hyperbolic graphs}, Journal of Graph Theory, vol. 57, pp. 157-180, 2008.

\bibitem{H0} E. Jonckheere, M. Lou, F. Bonahon and Y. Baryshnikov, {\it Euclidean versus hyperbolic congestion in idealized versus experimental networks}, http://arxiv.org/abs/0911.2538.

\bibitem{H2} D. Krioukov, F. Papadopoulos, A. Vahdat and M. Boguna, {\it Curvature and temperature of complex networks}, Physical Review E, 80:035101(R), 2009.

\bibitem{H3} P. Lohsoonthorn, {\it Hyperbolic Geometry of Networks}, Ph.D. Thesis, Department of Electrical Engineering, University of Southern California, 2003. Available at
http://eudoxus.usc.edu/iw/mattfinalthesis main.pdf.

\bibitem{H4} M. Lou, {\it Traffic pattern analysis in negatively curved networks}, Ph.D. Thesis, USC, May 2008. Available at http://eudoxus.usc.edu/iw/Mingji-PhD-Thesis.pdf.

\bibitem{H5} O. Narayan and I. Saniee, {\it The large scale curvature of networks}, Available at http://arxiv.org/0907.1478.

\bibitem{Patterson} S. Patterson, {\it The limit set of a Fuchsian group}, Acta Math. 136, no. 3--4, 241--273, 1976.

\bibitem{Sullivan} D. Sullivan, {\it The density at infinity of a discrete group of hyperbolic motions}, Inst. Hautes Etudes Sci. Publ. Math., no. 50, 171--202, 1979.

\bibitem{Ya} Ya. B. Pessin, {\it Dimension theory in dynamical systems}, Chicago Lect. Notes in Math., 1997

\end{thebibliography}
\end{document}